\documentclass{amsart}

\usepackage{amsthm}
\usepackage{amscd} 
\usepackage{amssymb}  
\usepackage{tikz}

\allowdisplaybreaks

\numberwithin{equation}{section} 

\newtheorem{Thm}{Theorem}[section] 
\newtheorem{Prop}[Thm]{Proposition} 
\newtheorem{Lem}[Thm]{Lemma} 
\newtheorem{Cor}[Thm]{Corollary} 
 
\theoremstyle{remark} 
\newtheorem{Rem}[Thm]{Remark}

\theoremstyle{definition} 
\newtheorem{Def}[Thm]{Definition} 
\newtheorem{Exa}[Thm]{Example} 
 
\newtheorem{Conv}[Thm]{Convention}

\newcommand{\R}{\mathbb{R}}
\newcommand{\Z}{\mathbb{Z}}
\newcommand{\N}{\mathbb{N}}
\newcommand{\Q}{\mathbb{Q}}
\newcommand{\C}{\mathbb{C}}
\newcommand{\E}{\mathcal{E}}
\newcommand{\lrh}{\longrightarrow}
\newcommand{\str}{\stackrel}
\newcommand{\ot}{\otimes}
\newcommand{\om}{\omega}
\newcommand{\wt}{\widetilde}
\usetikzlibrary{calc,through}

\begin{document}

\author{Paolo Salvatore}
\address{Dipartimento di Matematica,  
Universit\`{a} di Roma Tor Vergata, 
Via della Ricerca Scientifica, 
00133 Roma, Italy}
\thanks{This work was partially supported by the Tor Vergata University grant E82F16000470005}

\title{Non-formality of planar configuration spaces in characteristic two}

\begin{abstract}
We prove that the ordered configuration space of 4 or more points in the plane has a non-formal singular cochain algebra
in characteristic two. This is proved by constructing an explicit non trivial obstruction class in the Hochschild cohomology of the cohomology ring of the configuration space, by means of the Barratt-Eccles-Smith simplicial model. We also show that if the number of points does not exceed its dimension, then an euclidean configuration space is intrinsically formal over any ring. 
\end{abstract}

\maketitle

\section{Introduction}
The notion of formality of a topological space $X$ is usually introduced
in the sense of rational homotopy theory via
some commutative differential algebra model in characteristic 0 of $X$, like the Sullivan-deRham algebra $A_{PL}(X)$, or the algebra of differential forms if $X$ is a manifold.
However one can define formality over any coefficient ring $R$  
in the non-commutative context, by requiring the algebra of $R$-valued cochains $C^*(X,R)$ to be quasi-isomorphic to the cohomology 
$H^*(X,R)$.
This means that $C^*(X,R)$ is connected to its cohomology $H^*(X,R)$ by a zig-zag of homomorphisms of differential graded associative $R$-algebras
 inducing isomorphisms in cohomology. If $R$ is a field then this property depends only on the characteristic of $R$.
If $R$ is a field of characteristic 0 then the formality in the associative sense is equivalent to the usual commutative formality by a recent result 
of Saleh \cite{Saleh}.

Let us consider the euclidean configuration spaces 
$$F_k(\R^n) = \{(x_1,\dots,x_k) \in (\R^n)^k \,  | \,  x_i \neq x_i \text{ for } i \neq j \} .$$
 
\noindent Kontsevich and Lambrechts-Volic  have proved their formality in characteristic zero.

\begin{Thm} \cite{Kon,LV}  \label{konlamb}
The configuration space $F_k(\R^n)$ is formal over $\R$ for any $k,n$.
\end{Thm}

The case of the configuration spaces in the plane
$F_k(\R^2)$ had been proved much earlier by Arnold \cite{Arn}.

What can be said about formality over the integers, or in positive characteristic? 

In section \ref{tre} we prove the following positive result. 

\begin{Thm} \label{intr}
The configuration space $F_k(\R^n)$ is intrinsically formal over any commutative ring $R$ if $n \geq k$.
\end{Thm}

This means that any space with the same $R$-cohomology ring as the configuration space is formal over $R$.
The case $n < k$ is open in general. A special case is 
$$F_3(\R^2) \simeq S^1 \times (S^1 \vee S^1)$$
 that is formal over $\Z$.

One might expect that formality over $\Z$ holds for all configuration spaces as in  the rational case.

However we found the following surprising result:

\begin{Thm} \label{main}
The configuration space $F_k(\R^2)$ is not formal over $\Z_2$ for any $k \geq 4$.
\end{Thm}

This implies immediately the non-formality over the integers.

\begin{Cor} \label{overz}
$F_k(\R^2)$ is not formal over $\Z$ for any $k \geq 4$.
\end{Cor}

  We approach this question by obstruction theory following
the work by Halperin-Stasheff \cite{HS} for cdga (commutative differential graded algebras) 
in characteristic 0 and its extension to the non-commutative case by El Haouari \cite{EH}.
We use the Barratt-Eccles-Smith simplicial model for the configuration space, and construct an explicit filtered 
model using its combinatorics. The obstruction class is not trivial in Hochschild cohomology. 
We had a computer-aided proof in an early version, but now we have a proof that can be checked directly.
In a sequel we will consider configuration spaces in $\R^3$.

\

The paper is organized as follows: in section 2 we recall the presentation of the cohomology rings of the euclidean configuration spaces and of their Koszul dual, the Yang Baxter algebras.
In section 3 we define the concept of formality and recall the theory of filtered models by Halperin-Stasheff .  
We then prove Theorem \ref{intr} using their theory.
In section 4 we describe the  Barratt-Eccles-Smith simplicial models for the euclidean configuration spaces. 
In section 5 we start constructing a filtered model for the Barratt-Eccles model for $F_4(\R^2)$ and consider the obstruction class to its formality.
In section 6 we prove that the obstruction class is not trivial in Hochschild cohomology, and so 
$F_k(\R^2)$ is not formal over $\Z_2$ for $k \geq 4$ ( Theorem \ref{main}).

\section{Cohomology of configuration spaces and Yang-Baxter algebras} \label{due}

We recall the presentation of the cohomology ring of the configuration spaces 
$F_k(\R^n)$. Then we describe its Koszul dual, the Yang-Baxter algebra, and 
its geometric interpretation. 

\

Consider the direction map
$\pi_{i,j}:F_k(\R^n) \to S^{n-1}$ from the $i$-th to the $j$-th particle 
given by $\pi_{i,j}(x_1,\dots,x_k)=(x_j-x_i)/|x_j-x_i|$.

Let $\iota \in H^{n-1}(S^{n-1})$ be the standard fundamental class.
We write $$A_{ij}=\pi_{i,j}^*(\iota) \in H^{n-1}(F_k(\R^n)).$$

The following computation of the cohomology ring is due to Arnold in the case
$n=2$ and Fred Cohen for $n>2$.

\begin{Thm} {\rm(Arnold, Cohen)} \cite{Arn, LNM} 
The cohomology ring $H^*(F_k(\R^n))$ for $n \geq 2$ has a presentation with $(n-1)$-dimensional generators
$A_{ij}$  for $1 \leq i \neq j \leq n$ and relations
\begin{enumerate}
\item $A_{ij}=(-1)^n A_{ji}$
\item $A_{ij}^2 =0$
\item  $A_{ij}A_{jk}+A_{jk}A_{ki}+A_{ki}A_{ij}=0$ for $i \neq j \neq k \neq i$ \rm{(Arnold)}
\end{enumerate}
\end{Thm}

Thus up to the grading this ring depends only on the parity of $n$.

\begin{Cor} \label{basis}
The cohomology groups $H^*(F_k(\R^n))$ are torsion free, and have a graded basis provided by the set
$$\{A_{i_1 j_1} \cdots A_{i_l j_l} \, | \, j_1< \dots <j_l\, , \, i_t <j_t \quad \forall  t \}$$
\end{Cor}

\begin{Cor}
The Poincare series of the configuration space is 
$$P(F_k(\R^n))=\prod_{m=1}^{k-1}(1+mx^{n-1})$$
\end{Cor}
Notice that $F_k(\R^n)$ is the total space of a tower of fibrations with 
fibers homotopic to $\vee_m S^{n-1}$ for $m=1,\dots,k-1$. Thus the additive structure does not see the twisting of the fibrations, but the cohomology ring does.

\
A remarkable fact is that the cohomology of the configuration spaces is a Koszul algebra ( see example 32 in \cite{Berglund}) .
The following holds with coefficients in $\Z$.

\begin{Prop}  \label{ybb}
 The Koszul dual algebra of $H^*(F_k(\R^n))$ for $n \geq 2$  is 
 the Yang-Baxter algebra $YB_k^{(n)}$ generated by classes
  $B_{ij}$ of degree $n-2$, for $1 \leq i \neq j \leq k$ under the
relations
\begin{enumerate}
\item  $B_{ij} = (-1)^nB_{ji}$ 
\item  $[B_{ij},B_{jk}] = [B_{jk}, B_{ki}] = [B_{ki},B_{ij}] \text{ for }  i \neq j \neq k \neq i$ {\rm (Yang-Baxter) } 
\item  $[B_{ij}, B_{rt}] = 0$ when $\{i,j\} \cap \{r,t\} = \emptyset$
\end{enumerate}
\end{Prop}

Clearly this algebra up to the grading depends only on the parity of $n$.

\begin{Prop} \cite{CG}
The Yang-Baxter algebra is torsion free  
and has a graded basis 
$$\{B_{i_1 j_1} \cdots B_{i_l j_l}\, | \, j_1 \leq \dots \leq j_l \, , \, i_t<j_t \; \forall t     \} $$
\end{Prop}

The Yang-Baxter algebra has the following geometric meaning for $n>2$: it is  the Pontrjagin ring of the loop space
of the configuration space $F_k(\R^n)$. Notice that under this hypothesis $F_k(\R^n)$ is $(n-2)$-connected.

Let $A_{ij}^* \in H_{n-1}( F_k(\R^n))$ be the homology basis dual to $A_{ij}$.
By the Hurewicz theorem and adjointness we have isomorphisms
$$H_{n-2}(\Omega F_k(\R^n)) \cong \pi_{n-2}(\Omega F_k(\R^n)) \cong
\pi_{n-1}(F_k(\R^n)) \cong H_{n-1}(F_k(\R^n))$$
Let $$B'_{ij} \in H_{n-2}(\Omega F_k(\R^n))$$ be the class corresponding 
to $A_{ij}^*$ under the composite isomorphism.

\begin{Thm} (Cohen-Gitler, Fadell-Husseini) \cite{CG, FH}
There is an isomorphism of algebras
 $$H_*(\Omega F_k(\R^n)) \cong YB_k^{(n)}$$
  for $n>2$ 
 sending $B'_{ij}$ to $B_{ij}$.
\end{Thm}

\begin{Cor}
The Poincare series of $\Omega F_k(\R^n)$ for $n>2$ is 
$$P(\Omega F_k(\R^n)) = \prod_{m=1}^{k-1}(1-mx^{n-1})^{-1}$$
\end{Cor}

This follows also from the homotopy equivalence
$$\Omega F_k(\R^n) \simeq \prod_{m=1}^{k-1} \Omega(\vee_1^m S^{n-1}) $$
that is not a loop map in general.

\
For $n=2$ the (ungraded) algebra $YB_k=YB_k^{(2)}$ has also a similar geometric meaning that we explain. Let us start from the following well known result.
\begin{Prop}
The configuration space $F_k(\R^2)$ is a classfying space of the pure braid group on $k$ strands $P\beta_k$.
\end{Prop}
Now consider the descending central series of $P\beta_k$ defined by 
$$G_1^k=P\beta_k \text{ and } G_i^k=[G_1^k,G_{i-1}^k].$$ 
The direct sum of the subquotients forms a Lie algebra $$\mathcal{L}_k= \oplus_i (G_i^k / G_{i+1}^k)$$
under the bracket induced by taking commutators.
Let $U$ denote the universal enveloping algebra functor.
\begin{Thm}  \label{descend}   \cite{CG}   
There is an isomorphism 
$YB_k \cong U\mathcal{L}_k$ sending $B_{ij}$ to the Artin generator $\mathcal{A}_{ij}$  of the pure braid group.  
\end{Thm}

\section{Formality and filtered models} \label{tre} 

We present the definition of formality in rational homotopy theory and in the non-commutative sense.
We then proceed to describe the non-commutative versions of the filtered models by Halperin and Stasheff, due to El Haouari.

\begin{Def} \cite{HS}
A 
topological space $X$ is rationally {\em formal}  if there 
is a zig-zag of quasi-isomorphisms of commutative differential graded algebras connecting the Sullivan-deRham algebra $A_{PL}(X)$  to 
its cohomology $H^*(A_{PL}(X)) \cong H^*(X,\Q)$ equipped with the trivial differential.
\end{Def}

If $X$ is a manifold this is equivalent to the existence of a 
zig-zag of quasi-isomorphisms between the algebra of de-Rham forms $\Omega^*(X)$ 
and its real cohomology $H^*(X,\R)$. If $X$ is a complex manifold we might use the ring of complex differential forms and
cohomology with complex coefficients. 
Namely the notion of formality in characteristic zero does not depend on the field for connected spaces of finite type (Theorem 6.8 in \cite{HS}).
The same is true in positive characteristic, by a similar proof following the obstruction theory in \cite{EH}.

\

 Kontsevich \cite{Kon} and Lambrechts-Volic \cite{LV} have proved that
the configuration space $F_k(\R^n)$ is formal over $\R$ for 
any $k$ and $n$ (Theorem \ref{konlamb} ).

This approach uses the commutative algebra of real PA forms and proves also the formality of the little discs operads.

\
Arnold had easily proved the formality of $F_k(\C)$  \cite{Arn}. His quasi-isomorphism embeds $H^*(F_k(\C),\C)$ into the algebra of 
complex differential forms $\Omega^*_{\C}(F_k(\C))$  sending $A_{ij}$ to $d(z_i-z_j)/(z_i-z_j)$.

We turn now to the non-commutative case.
\begin{Def}
Let $R$ be a commutative ring. A topological space $X$ is
$R$-formal, or formal over $R$, if the algebra of singular cochains $C^*(X,R)$ is
connected to its cohomology $H^*(X,R)$ by a zig-zag of quasi-isomorphisms of differential graded $R$-algebras. 
\end{Def}

For spaces of finite type with torsion free homology, $\Z$-formality is universal as it implies $R$-formality for any ring $R$.

Here is an application of formality to computations:
let $\Omega X$ and $LX$ be respectively the based and the unbased loop space of $X$. For a graded $R$-algebra $A$ let us denote
by $HH_*(A,A)$ its Hochschild homology with coefficients in itself, that is a graded $R$-module. 

\begin{Prop}
Let $X$ be a simply connected $R$-formal space of finite type.
Then there are isomorphisms of 
graded $R$-modules
\begin{align*}
Tor_{H^*(X,R)}(R,R)  \cong  H^*(\Omega X, R) \\
HH_*(H^*(X,R),H^*(X,R)) \cong H^*(LX,R)
\end{align*}
\end{Prop} 
\begin{proof}
The classical work of Eilenberg-Moore shows that 
$$Tor_{C^*(X,R)}(R,R) \cong H^*(\Omega X,R), \, {\rm and}$$ 
$$HH_*(C^*(X,R),C^*(X,R)) \cong H^*(LX,R).$$ Since $X$ is formal we can replace $C^*(X,R)$ by $H^*(X,R)$ in the formula.
\end{proof}

The theory of $R$-formality is studied by El Haouari \cite{EH}, who extends the obstruction theory by Halperin-Stasheff to the non-commutative case.
The key point is the existence of a {\em bigraded model} for the cohomology of an algebra that is then deformed to a {\em filtered  model} of the algebra.
A bigraded module $V=\oplus V^n_k$ has an upper dimensional grading, the {\em degree} $n$, and a lower grading $k$, the resolution {\em level}. 
\begin{Conv}
We work in the category of 
{\em cochain} differential graded algebras ($R$-DGA's) over a commutative ring $R$ with upper grading, such that the differential is homogeneous of degree $1$.
A bigraded $R$-algebra has a differential that is also homogeneous with respect to the lower grading, and lowers it by $1$. 
The tensor algebra on a bigraded module inherits a bigrading as well, and its cohomology too.

 We say that a $R$-DGA $A$ is connected if $H^0(A) \cong R$ is generated by the class of the unit $1_A \in A$. 
In particular a graded $R$-algebra $A$, considered as DGA with trivial differential, is connected when $A^0=R \cdot 1_A \, .$
\end{Conv}

\begin{Prop}  {\rm (2.1.1.of \cite{EH})} \label{res}
Let $k$ be a field. Given a connected graded $k$-algebra $H$, there exists a bigraded $k$-module $V$ 
and a differential $d$ on the tensor algebra $T(V)$ together with a 
quasi-isomorphism 
$$\rho: (T(V),d) \to (H,0)$$
 such that the homology in positive resolution level vanishes, i.e.
$H_+(T(V),d)=0, $ and
$\rho_{|V_0}: V_0 \cong H^+/(H^+\cdot H^+) \to  H^+$ is a splitting of the projection to the indecomposables. 
The algebra $(T(V),d)$ is unique up to isomorphism and it is called the {\em bigraded model} of $H$.
\end{Prop}

The main example that we need to consider is the bigraded model of the cohomology of the euclidean configuration spaces.

\begin{Exa} \label{boh}
The configuration space cohomology ring $H=H^*(F_k(\R^n),R)$ admits a bigraded model over any commutative ring $R$
in the sense of Proposition \ref{res}.  The model is constructed over $\Z$. Tensoring with $R$ gives the general case.
The Yang-Baxter algebra $YB_k^{(n)}=H^!$ admits a bigrading. The upper grading was described in Proposition \ref{ybb}.
The lower grading is the length of words in the standard generators $B_{i,j}$  minus one. 
We have that $YB_k^{(n)}$ has finite type as a bigraded algebra, i.e. it is finite dimensional in each bidegree. The algebra has not finite type with respect to the upper grading  for $n=2$. Let $\wt{YB_k^{(n)}}$ be the ideal of positive degree elements (we just discard the 0-dimensional part spanned by the unit).
Its bigraded dual $R$-module $V=s\wt{YB_k^{(n)}}^*$, suspended in the upper grading by 1, generates the free bigraded algebra $(T(V),d)= B(YB_k^{(n)})^*$ that is the bigraded dual of the bar construction. The differential $d$ is the derivation of $T(V)$ 
induced by the coproduct $V \stackrel{\mu^*}{\longrightarrow} V \otimes V \subset  T(V)$  
dual to the multiplication $\mu$ of $YB_k^{(n)}$.
The resolution map $\rho: (T(V),d) \to H$ is then defined by
$$\rho(B_{ij}^*)=A_{ij},$$ 
$$\rho((B_{i_1 j_1} \dots B_{i_l j_l})^*)=0  \quad  \forall \,  l>1$$
Since $H$ is Koszul, $\rho$ is a quasi-isomorphism. 

\end{Exa}

\
A bigraded module $A$ has an induced filtration defined by $F_k(A) = \oplus_{i \leq k}  A^*_i$ .
We say that a differential $d$ on $A$ lowers the filtration level by $h$ if $d(F_k(A)) \subset F_{k-h}(A)$.

The key idea of Halperin-Stasheff is to deform the bigraded model of the cohomology $H(A)$ of a DGA $A$ in order to get a {\em filtered model} for the
algebra itself.

\begin{Thm} {\rm (2.2.2 in \cite{EH})}  \label{222}
Let $A$ be a connected DGA over a field $k$. 
 Let $$\rho:(T(V),d) \to H^*(A)$$ be the bigraded model. Then there is a differential $D$  on $T(V)$ and
a quasi-isomorphism
$$\pi: (T(V),D) \longrightarrow A$$  such that 
\begin{itemize}
\item $D-d$ lowers the filtration level by 2;
\item for $x \in V_0, \, [\pi(x)]=\rho(x) \in H^*(A)$. 
\end{itemize}
 If $\pi':(T(V),D') \to A$ is another solution to this problem, then 
there exists an isomorphism $$\phi: (T(V),D) \longrightarrow (T(V),D')$$ such that 
$\phi-id$ lowers the filtration level by 1, and $\pi' \circ \phi \simeq \pi$. Any solution is called a {\em filtered model} of $A$. 
\end{Thm}

The following result completes the theory, compare Theorem 5.3 in \cite{HS}.
\begin{Prop}  \label{341}
Let $A,B$ be connected DGA's over a field $k$ with an isomorphism $\bar{\phi}:H(A) \cong H(B)$, and let
$\rho: (T(V),d) \to H(A) \cong H(B)$ be the bigraded model. 
Let $(T(V),D_A), (T(V),D_B)$ be the respective filtered models of $A$ and $B$. Then $A$ and $B$ are quasi-isomorphic
if and only if there exists an isomorphism $\phi:(T(V),D_A) \to (T(V),D_B)$ such that $\phi-id$ lowers the filtration level. 
\end{Prop}

We give a first immediate application.

\begin{Def}
A space $X$ is intrinsically $R$-formal if it is $R$-formal, and  any other space with the same cohomology $R$-algebra is also $R$-formal.
\end{Def}

\begin{Thm}
The configuration space $F_k(\R^n)$ is intrinsically formal over any commutative ring $R$ if $n \geq k$.
\end{Thm}

\begin{proof}

The generators of the bigraded model $(T(V),d)$ from example \ref{boh}   for $H=H^*(F_k(\R^n),R)$ have degree (upper index) $$n-1,2n-3, \dots, 1 + i(n-2), \dots  $$
since $V \cong s\wt{(YB_n)}^*$, 
and have respective level (lower index) $$0,1,2,3,\dots, i-1, \dots $$
Let us consider the filtered model $(T(V),D)$ for the singular cochain algebra $C^*(Y,R)$ of a space $Y$ with 
$H^*(Y,R) \cong H$.  This exists over $\Z$ and is obtained deforming $d$ as in the proof of Theorem 4.4 in  \cite{HS} .  
The tensor product with $R$ gives a model over $R$.
We have that $D-d$ lowers the filtration level by 2.
All generators are in degree $1$ mod $n-2$, and so their differentials are in degree $2$ mod $n-2$, that is $2,n,2n-2, \dots$
For a generator $v$ of level $i-1$, and thus degree $1+i(n-2)$, a homogeneous non-zero monomial can be a term of $(D-d)(v)$  only if it has degree 2 mod $n-2$. 
But quadratic monomials $v_a v_b$, with $v_a \in V_{a-1}$ and $b \in V_{b-1}$,  have to be ruled out because they would have degree 
$$2+i(n-2)=(1+a(n-2))+(1+b(n-2)),$$
 thus filtration level
 $(a-1)+(b-1)=i-2$, contradicting that $D-d$ lowers the filtration level by two. The monomial is then the product of at least $2+(n-2)=n$ generators, that is in 
 degree at least $$n+n(n-2) = 2+(n+1)(n-2),$$ so $i \geq n$. Then $d=D$ on $V_{< n}$.
 The top cohomology of the configuration space (and $Y$) is in dimension $(k-1)(n-1)<n(n-1)$.
 The latter is the upper degree of $dx$, for $x \in V_n$, that is less than or equal to the degree of $dx$, for $x \in V_m$ and  $m \geq n$.
 So we can extend the partial resolution $(T(V_{< n}),D) \to H$ by the trivial homomorphism $V_{\geq n} \to H$ to a quasi-isomorphism
$(T(V),D) \to H$. 

\end{proof}

\begin{Rem}
A special case not covered by the theorem
 is that of $F_3(\R^2) \simeq S^1 \times (S^1 \vee S^1)$, which is 
$\Z$-formal , (and therefore $R$-formal for any commutative ring $R$) since wedges and products preserve formality \cite{HS}. 
\end{Rem}

In order to prove the non-formality of $F_k(\R^2)$ (Theorem \ref{main}) over $\Z_2$ for $k>3$ we introduce simplicial versions of these configuration spaces in the next section.

\section{The Barratt-Eccles complex} \label{quattro}

We recall the Smith simplicial model for planar configuration spaces, obtained by filtering the Barratt-Eccles complex \cite{BF}.
For simplicity we write $F_k:=F_k(\R^2)$, and we adopt $\Z_2$-coefficients throughout although we do not state it explicitly.

\begin{Def}
The full Barratt-Eccles complex on $k$ elements is 
 the geometric bar construction of the symmetric group on $k$ letters $W\Sigma_k$.

This is a simplicial set with 
$(W\Sigma_k)_l = (\Sigma_k)^{l+1}$, so any element is a string of permutations.  
The face operators delete a single permutation in the string, and the degeneracies double a permutation in the string.

The symmetric group $\Sigma_k$ acts diagonally levelwise.
\end{Def}

In particular $W\Sigma_2$ has 2 non-degenerate simplices in each degree
that are 
$$(12|21|12|..), \quad (21|12|21|..).$$
Consider the sub-simplicial set $\mathcal{F}_t(W\Sigma_2) \subset W\Sigma_2$ spanned
by non-degenerate simplices of dimension at most $t-1$.

The geometric realization $|\mathcal{F}_t(W\Sigma_2)| $ has the homotopy type of 
 the $(t-1)$-sphere.

 We can extend the filtration to any $k>2$ as follows.
 
There are simplicial  versions of the projections seen in section \ref{due}, that we still denote
$$\pi_{ij}:W\Sigma_k \to W\Sigma_2, \text{ for } 1 \leq i \neq j \leq k .$$

Levelwise these maps are powers of the function $\pi_{ij}:\Sigma_k \to \Sigma_2$  defined by 
$$
\pi_{ij}(\sigma)=\begin{cases}
(12) \text{ if }\sigma^{-1}(i)<\sigma^{-1}(j) \\
(21) \text{ if }\sigma^{-1}(i)>\sigma^{-1}(j). 
\end{cases} $$

\begin{Def}
The $t$-th stage 
$\mathcal{F}_t(W\Sigma_k)$
of the filtration of the Barratt-Eccles complex is defined by

$$\mathcal{F}_t(W\Sigma_k)_l=\bigcap_{i,j} \pi_{ij}^{-1}(\mathcal{F}_t(W\Sigma_2)_l)$$
\end{Def}

\begin{Prop} \label{kashi} \cite{BF}    
The geometric realization of the simplicial set 
$\mathcal{F}_t(W\Sigma_k)$ has the homotopy type
of the configuration space $F_k(\R^t)$ ,
and the realization of $\pi_{ij}$ up to homotopy corresponds to the projection $\pi_{ij}:F_k(\R^t) \to S^{t-1}$ considered in section \ref{due}.
\end{Prop}

Let $N^*$ denote the normalized cochain functor from simplicial sets to differential graded algebras (over $\Z_2$ in our case).

\begin{Def}
The Barratt-Eccles  DG-algebra with {\em complexity} $t$ and {\em arity} $k$ is 
$$\mathcal{E}_t ^*(k):= N^*(\mathcal{F}_t(W\Sigma_k)).$$
\end{Def}

Here $k$ is the {\em arity} of the Barratt-Eccles operad, that is equal to the number of points in the configuration.
The upper index indicates the degree.
We are interested mainly in $t=2$: a string of permutations is a generator of $\mathcal{E}_2$ if any two indices do not swap more than once.
For example $(1 2 3| 1 3 2| 32 1)^* \in \mathcal{E}_2^2(3)$ but $(1 2 3| 3 2 1| 2 3 1)^* \notin \mathcal{E}_2^2(3)$ because the indices $2$ and $3$ swap twice.

\
Observe that each $\mathcal{E}^i_t(k)$   is a free module over $\Z_2[\Sigma_k]$, so that $k!$ divides its 
dimension over $\Z_2$ (this holds for any coefficient ring).
For fixed $t$ and $k$ the top dimensional non-trivial vector space is in degree $i=(t-1)\binom{k}{2}$. 

To get a sense of the numbers involved we present some cases of the generating polynomial
$$P_t^k(x) = \sum_i  dim_{\Z_2}(\E^{i}_t(k)) x^i $$  

\begin{align*}
& P_2^2 (x) = 2(1+x)  \\
& P_2^3(x)= 6(1+5x+6x^2+2x^3)\\
& P_2^4(x)=24(1+23x+104x^2+196x^3+184x^4+86x^5+16x^6)\\
& P_3^2(x)=2(1+x+x^2) \\
& P_3^3(x)=6(1+5x+25x^2+60x^3+70x^4+38x^5+8x^6 ) \\
& P_3^4(x)=24(1+23x+529x^2+5550x^3+30214 x^4 + 97048 x^5 +    \dots    )  
\end{align*}

Let us denote the elements of the symmetric group $\Sigma_3$ 
by the following letters:
$$A=(123), B=(132), C=(213), D=(231), E=(312), F=(321).$$

Then $\mathcal{E}_2^2(3)$ is the free $\Z_2[\Sigma_3]$-module on the 6 simplices
$$(A|D|F)^*, \, (A|E|F)^*, \, (A|C|F)^*, \, (A|C|D)^*, \, (A|B|F)^*, \, (A|B|E)^* ,$$
hence it is 36-dimensional over $\Z_2$.

By Proposition \ref{kashi}
we have quasi-isomorphic DGA's
$\mathcal{E}_2^*(k) \simeq C^*(F_k(\R^2))$.

\begin{Prop}
The non-formality of $\E_2^*(k)$ for $k \geq 4$ is equivalent to the statement of Theorem \ref{main}.
\end{Prop}

For $1 \leq i \neq i \leq k$ consider the 1-cochain
$$\omega_{ij} = \pi_{ij}^*((12|21)^*) \in \mathcal{E}^1_2(k).$$

For example for $k=3 \;\omega_{ij}$ is the sum of 9 generators, and for $k=4$ it is the sum of 144 generators.

Consider the 1-cochain $$Ar:=(B|E)^* \,  \in \mathcal{E}^1_2(3)  $$

\begin{Lem} 
A resolution of the Arnold relation for 3 points on the cochain level is given by
$$d(Ar)=\omega_{13}\omega_{12}+\omega_{23}\omega_{12}+
\omega_{23}\omega_{13}$$
\end{Lem}

\begin{proof}
$d(Ar)$ is the sum of all cochains of the form $(B|E|x)^*, \, (x|B|E)^*$ and $(B|x|E)^*$ in complexity 2 .
It is not difficult to check that
\begin{equation} \label{dar}
 d(Ar)=   (B|E|D)^*+(B|E|F)^*+(A|B|E)^* +(C|B|E)^*  \, .
 \end{equation}
The cochain $\omega_{13} \omega_{12}$ pairs non trivially with chains $(x|y|z)$ with $(13)$ and $(12)$ substrings of $x$, $(31)$ and $(12)$ of $y$, and 
$(31)$ and $(21)$ of $z$. This forces $x=(123)$ or $x=(132)$. In the first case $y=(312)$ and $z=(321)$. In the second case $y=(312)$, and either $z=(321)$ or $z=(231)$, hence
\begin{equation} \label{eq1}
\omega_{13}\omega_{12} = (A|E|F)^* + (B|E|F)^*+(B|E|D)^* . 
\end{equation}
By similar considerations we obtain  
\begin{equation} \label{eq2}
\omega_{23}\omega_{12}= (A|B|F)^* +(A|E|F)^* ,
\end{equation}
\begin{equation} \label{eq3}
\omega_{23}\omega_{13} = (A|B|E)^* + (A|B|F)^*+(C|B|E)^* . 
\end{equation}
 Comparing the sum of equations \ref{eq1}, \ref{eq2}, \ref{eq3} with equation \ref{dar} we conclude.

\end{proof}

\section{The filtered model of the Barratt Eccles complex} \label{cinque}

In this section we start building the filtered model of the Barratt-Eccles algebra $\E_2^*(4)$ and verify its non-formality.
Let $(TW,D)$ be the filtered model of $\E_2^*(4)$ ( Theorem \ref{222}),
  obtained by deforming the bigraded model $(TW,d) \simeq H^*(F_4)$, and equipped with a quasi-isomorphism $\phi: (TW,D) \to \E_2^*(4)$.

By Theorem \ref{222}, $d=D$ on $W_0$ and $W_1$. We can characterize it explicitly.

\begin{Lem} \label{dw2}
On $W_0, \; D=d=0$. On $W_1, \; D=d$ is determined by 

$$d((B_{ij}B_{kl})^*)=\begin{cases}
(B_{ij}^*)^2  \quad{\rm    if }\; i=k \;{\rm and }\; j=l \\
B_{ij}^*B_{kl}^*+B_{kl}^*B_{ij}^*=[B_{ij}^*,B_{kl}^*] \quad{\rm for }\quad   j < l \\
B_{ij}^*B_{kl}^*+ B_{il}^*B_{ki}^*+B_{kl}^*B_{ki}^*  \quad{\rm for }\;   j=l \;{\rm and }\; k<i \\
B_{ij}^*B_{kl}^*+ B_{il}^*B_{ik}^*+B_{kl}^*B_{ik}^* \quad{\rm for }\; j=l \;{\rm and }\; k>i
\end{cases}
$$

\end{Lem}

\begin{proof}
The product of two generators $B_{ij}B_{uv}$ of the Yang-Baxter algebra is a standard basis generator if and only if $j \leq v$.
Instead, for $j >v$, the Yang Baxter relations yield
$$B_{ij} B_{uv}=\begin{cases}
 B_{uv} B_{ij}   \quad{\rm if }\; \{i , j \} \cap \{u,v\} = \emptyset \\
B_{uv}B_{ij} + B_{uj}B_{vj}+ B_{vj}B_{uj} \quad {\rm otherwise} 
\end{cases}
$$
By dualizing the formula on basis generators we obtain the result.

\end{proof}

We shall show that $d$ and $D$ do not agree on $W_2$, and there is no ``gauge equivalence'' transformation 
fixing the problem (compare Proposition \ref{341}).

Following the construction of the filtered model in \cite{EH,HS} (mentioned in Theorem \ref{222})
we can assume that $\phi$ sends each  generator in $W_0 \cong H^1(F_4)$  to an arbitrary cocyle representative.

\begin{Def} \label{zero}
Let $\phi: W_0 \to \E_2^1(4)$ 
be the linear map given by  $$\phi(B_{ij}^*)= \omega_{ij}\, .$$
\end{Def}

We move to the next set of generators in filtration 1.  
\begin{Def} \label{one}
Let $\phi:W_1 \to \E_2^1(4)$ be the linear map given by
$$\phi((B_{ij}B_{kl})^*)= \begin{cases} \label{quadratic} 
0 \text{ if } (i,j)=(k,l)  \\
\omega_{ij} \cup_1 \omega_{kl}  \text{ if  } j<l \\ 
\pi_{kil}^*(Ar) \text{ if } i>k \text{ and }j=l \\
\pi_{ikl}^*(Ar)+ \omega_{il}\cup_1 \omega_{kl} \text{ if } i<k \text{ and } j=l .
\end{cases} $$
\end{Def}
\
Let us explain the notation. The construction of $\cup_1$ goes back to Steenrod, and is presented by  
McClure-Smith in \cite{MS}. 
We remark that the $\cup_1$ product of 1-cochains  
over $\Z_2$ has a simple form.
Each $k$-cochain $c$ over $\Z_2$ is determined by its support $Supp(c)$, the set of non-degenerate $k$-chains sent to 1 by  $c$ (rather than 0).  Now for given 1-cochains $c,c'$ the relation
$Supp(c \cup_1 c')=Supp(c) \cap Supp(c')$ determines 
$c \cup_1 c'$.
\medskip

The homomorphism $\pi_{kil}:W \Sigma_4 \to W \Sigma_3$ is induced by the simplicial version of the forgetful projection from 4 to 3 configuration points that we describe below (the projection to 2 points has been described in section 4).

We need a function $\pi_{kil}:\Sigma_4 \to \Sigma_3$ measuring how a permutation twists the indices $\{k,i,l\}$.
Consider the three symbols $j_1=k,j_2=i,j_3=l$.
A permutation $\sigma \in \Sigma_4$ gives a sequence
$(\sigma(1), \sigma(2), \sigma(3), \sigma(4))$ containing a unique ordered subsequence $(j_{\tau(1)},j_{\tau(2)},j_{\tau(3)})$, where
$\tau \in \Sigma_3$. We set $\pi_{kil}(\sigma)=\tau$.
The construction induces a map on the W-construction still denoted $\pi_{kil}:W\Sigma_4 \to W\Sigma_3$, and then on normalized cochains
$$\pi_{kil}^*=N^*(\pi_{kil}):\E_t^*(3) \to \E_t^*(4) \, .$$
For example we have the pullback of 0-cochains $$\pi^*_{123}((312)^*)=(4312)^*+(3412)^*+(3142)^*+(3124)^* \, .$$

\begin{Lem}
The homomorphism $\phi$ of Definition \ref{zero} and  \ref{one} is compatible with the differentials,
and thus defines a homomorphism of DGA's  
$$\phi: (T(W_{\leq 1}),d) \to \E_2^*(4).$$
\end{Lem}

\begin{proof}

We have that $d((B_{ij}^2)^*)=(B_{ij}^*)^2$ by lemma \ref{dw2}

In the first case of Definition \ref{quadratic} $$\phi ( d((B_{ij}^2)^*))=
 \phi((B_{ij}^*)^2)=\omega_{ij}^2=0=d(\phi ((B_{ij}^2)^*)).$$
 Namely  the top-dimensional 1-cocyle $(12|21)^* \in \E_2^1(2)$ squares to zero by dimensional reasons, and then so does its pullback $\omega_{ij} \in \E_2^1(3)$.
\
In the second case of Definition \ref{quadratic}
by Lemma \ref{dw2}
$$d((B_{ij}B_{kl})^*)=[B_{ij}^*,B_{kl}^*]  \text{ if } j<l , \text{ and}$$
$$\phi(d((B_{ij}B_{kl})^*))=
\phi([B_{ij}^*,B_{kl}^*])= \omega_{ij} \cup \omega_{kl}
+\omega_{kl} \cup \omega_{ij} =d(\omega_{ij}\cup_1 \omega_{kl})$$
by Steenrod's formula \cite{MS}. 
By definition $$d(\omega_{ij}\cup_1 \omega_{kl})=
d(\phi((B_{ij}B_{kl})^*)).$$

\
In the third case of Definition \ref{quadratic}, for  $j=l$ and $i>k$, we have from Lemma \ref{dw2} that
$$d((B_{ij}B_{kl})^*)= B_{ij}^*B_{kl}^*+ B_{ij}^*B_{ki}^*+B_{kj}^*B_{ki}^*$$
and $\phi$ maps the right hand side  to
$$\omega_{il}\omega_{kl}+\omega_{il}\omega_{ki}+\omega_{kl}\omega_{ki}=
\pi_{kil}^*(d(Ar))=d(\pi_{kil}^*(Ar))=d\phi((B_{ij}B_{kl})^*).$$

In the last case of Definition \ref{quadratic} 
for $i<k$ we have the same formula 
$$d((B_{ij}B_{kl})^*)= B_{ij}^*B_{kl}^*+ B_{ij}^*B_{ik}^*+B_{kj}^*B_{ik}^*$$
(we only swap $i$ and $k$ to have the term $B_{ik}$ in standard form),
 so
\begin{align*}
\phi(d((B_{ij}B_{kl})^*)) =
\omega_{il}\omega_{kl}+\omega_{il}\omega_{ik}+\omega_{kl}\omega_{ik}
=\pi_{ikl}^*(d(Ar))+[\omega_{il},\omega_{kl}] =	\\
=d(\pi_{ikl}^*(Ar)+\omega_{il}\cup_1 \omega_{kl}) =
d(\phi ((B_{ij}B_{kl})^*)).
\end{align*}

Of course the product of cochains is not commutative and this matters 
in the ordering of the indices in the Arnold relation.
\end{proof}

We compute the ``error term'' $D(w)-d(w) \in T(W_0)$, for $w \in W_2$, following the proof of Theorem \ref{222}. 

Since $d(w) \in T(W_{\leq 1})$, we have already defined $\phi(d(w))$, that is a cocyle, since $\phi$ commutes with $d$ so far.
Its cohomology class $[\phi(d(w))] \in H^2(\E_2^*(4)) \cong H^2(F_4)$  might not be trivial, preventing us from extending $\phi$ to $T(W_{\leq 2},d)$.

\begin{Def} \label{iota}
Let $\iota:H^2(F_4) \to T_2(W_0)$ be the embedding into the quadratic part sending each standard generator in the $A_{ij}$'s to the analogous standard generator in the
$B_{ij}$'s. We set then $D(w)=d(w)-\iota[\phi(d(w))]$. 
\end{Def}

\begin{Prop} \label{prop5.6}
There is an extension of $\phi: (TW_{\leq 2},D) \to \E_2^*(4)$ commuting with the differentials.
\end{Prop}

\begin{proof}
We need to define $\phi(w)$ for $w \in W_2$, as $w$ varies in the standard basis.
In cohomology  $[\phi(D(w))]=[\phi(d(w))]-[\phi(d(w)]=0$, since $[\phi(\iota(x))]=x$. So there exists a cochain $\phi(w)  \in \E_2^1(4)$ with 
$d\phi(w)=\phi(D(w))$.
 \end{proof}

The inductive construction continues in an infinite number of stages, building the full model $(TW,D)$, but here we are concerned only with filtration $\leq 2$.

Explicitly the ``error term'' $\phi(d(w))$ 
is the degree  one map 
\begin{equation} \label{formulona} \phi \circ d: W_2 \stackrel{\mu^*}{\longrightarrow} W_1 \otimes W_0  \oplus W_0 \otimes W_1  
\xrightarrow{  \phi_1 \ot \phi_0 + \phi_0 \ot \phi_1} \E^1 \ot \E^1    \str{\cup}{\lrh}  \E^2 \,  \end{equation}  
evaluated on $w$.

\begin{Def} \label{alpha}
For $w \in W_2$,  let us define the cohomology class $\alpha(w) := [\phi(dw)].$
This defines a homomorphism $\alpha:W_2 \to H^2(F_4)$. 
 \end{Def}

We need a partial computation of $\alpha$ .
The computation is achieved by evaluating the cocycles $\phi(d(w))$ on standard chain representatives for basis elements of the homology group
$H_2(F_4) \cong H_2(\mathcal{E}_2(4)^*)$ that has dimension 11. 
Each representing cycle will be the sum of 8 simplexes, that are generators of $\mathcal{E}_2^2(4)^*$. 
We need to recall some facts on operads with multiplication in order to describe the cycles.

\begin{Def} \cite{Deligne}
A differential graded operad (DGO) with multiplication over a fixed field $F$ is a sequence of $F$-vector spaces 
$O(k), \, k \in \N$,  endowed with composition operations $$O(k) \otimes O(l) \to O(k+l-1)$$ for $1 \leq i \leq k$, denoted 
$(x,y) \mapsto x \circ_i y$, and multiplication operations $$O(k) \otimes O(l) \to O(k+l)$$ denoted $(x,y) \mapsto x \cdot y$, satisfying appropriate axioms.
\end{Def}

\begin{Exa}
The Barratt Eccles operad $(\E_n)^*$ forms a DGO with multiplication (over $\Z_2$ in this paper) \cite{BF}.
Its homology 
\begin{equation}\label{theta}  H_*(\E_n(k)^*) \stackrel{\theta}{\cong} H_*(F_k(\R^n)) \end{equation}
 inherits a structure of graded operad with multiplication . 
If we forget the multiplication, then $(\E_n)^*$ is weakly equivalent to the chain operad of the little $n$-discs operad as a DGO \cite{BF}.
\end{Exa}

We use lemma 6 in \cite{palermo} that gives explicit representatives of the homology basis 
of the euclidean configuration spaces.

\begin{Lem} \cite{palermo}
Each generator of the basis of $H_{i(n-1)}(F_k(\R^n))$ dual to the cohomology basis in Corollary \ref{basis} is represented by an embedding
$(S^{n-1})^i  \to F_k(\R^n)$ with image the configuration space of an appropriate composite planetary system of $k$ bodies in $\R^n$.
 \end{Lem}
 
The operadic composition and the multiplication in the operad $H_*(F_*(\R^n))$ have an easy interpretation in terms of planetary systems.
Informally, if $x \in H_t(F_k(\R^n))$ and $y \in H_u(F_l(\R^n))$ are represented by planetary systems, then
the system of $x \circ_i y \in H_{t+u}(F_{k+l-1}(\R^n))$ is obtained replacing the $i$-th planet of $x$ by a small copy of the system of $y$, and relabelling the bodies involved. This procedure is similar to the operadic composition in the little discs operad. 
The system of the multiplication $x \cdot y \in H_{t+u}(F_{k+l}(\R^n))$ is the union of two copies of the system $x$ and $y$, far apart, with the labels of the bodies of $y$ raised by $k$. 

 Let us start from the homology generator  $a \in H_1(F_2)\cong H_1(S^1) \cong \Z_2$. 
This is represented by a configuration space of a planet labelled 2 rotating around a star labelled 1. 
 The generator $(A_{12}A_{23})^*$ of $H_2(F_3)$  is 
the operadic composition $a \circ_2 a$.  Its system is
the configuration space of a planetary system with a star labelled 1, a planet labelled 2, and a satellite labelled 3. 

\

The cycle $\gamma=(1 2 |  2 1)+(2 1 | 1 2) \in \E_2^1(2)^*$ represents the generating class $[\gamma] \in H_1(\E_2(2)^*)$ corresponding to $a$ under the isomorphism $\theta$ in equation \ref{theta}.
Let us compute the class $[\gamma \circ_2 \gamma]= [\gamma] \circ_2 [\gamma] \in H_2(\E_2(3))$ corresponding to $a \circ_2 a$ under $\theta$.
In the Barratt-Eccles operad the composition is defined by appropriate substitution and relabelling  \cite{BF}. For example 
$$(1 2 | 2 1)\circ_2 (2 1 | 1 2) = (1  3 2 | 1 2 3 | 2 3 1)+ 
(1 3 2 | 3 2 1| 2 3 1).$$
 This implies that the cycle $\gamma \circ_2 \gamma$ is a sum of 8 simplexes, namely

\begin{align*} 
\gamma \circ_2 \gamma =(132|123|231)+(123|132|321)+(123|231|321)+(132|321|312)+ \\
+(321|132|123)+(231|123|132)+(321|231|123)+(231|321|132).
\end{align*}
 
  The class  $(A_{12}A_{23})^* \in   H_2(F_4)$ is the multiplication $(a\circ_2 a) \cdot u $  by the base point generator $u \in H_0(F_1)$.
Its system has a star labelled 1, a planet labelled 2, a satellite labelled 3, and another star labelled 4 far away.

\

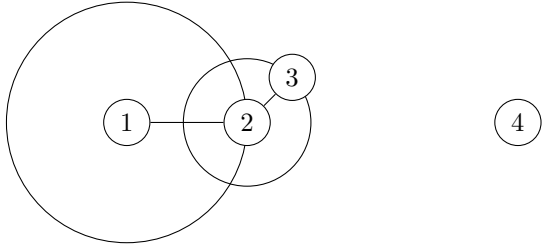
 \begin{tikzpicture} [scale=.4]
  \tikzstyle{star}=[draw,shape=circle,fill=white];
 \draw (0,0)circle (4); 
 \draw (4,0) circle (2.12);
 \node[star] (1) at (0,0) {1};
  \node[star] (2) at (4,0) {2};
  \node[star] (3) at (5.5,1.5) {3};
   \node[star](4) at (13,0) {4};

  \draw (1)--(2)
            (2)--(3);
  \end{tikzpicture}           
 
  \
  
In the Barratt Eccles operad the product $x  \cdot 1 \in \mathcal{E}(4)^*$ of $x \in \mathcal{E}(3)^*$ and of the $0$-chain $1 \in \mathcal{E}^0(1)^*$
is obtained by sticking a 4 at the end of each permutation of $x$. The effect of this operation on $\gamma \circ_2  \gamma$ gives the first representative 
\begin{align*} 
 T=(\gamma \circ_2 \gamma) \cdot 1 =(1324|1234|2314)+(1234|1324|3214)+(1234|2314|3214)+\\
 (1324|3214|3124)+
(3214|1324|1234)+(2314|1234|1324)+\\(3214|2314|1234)+(2314|3214|1324). 
 \end{align*}

The next homology generator $ (A_{12}A_{13})^*$ is dual to a system where a star is labelled 2, a planet 1, and a satellite 3. 
This homology class is the effect of the $(12)$-action on the previous class (exchanging 2 and 1). This makes 
sense in any operad, and on the chain level in the Barratt Eccles operad as well, yielding the second representative $(12)T$.

Proceeding in this way we obtain 8 basis elements out of 11, just by permuting labels. Namely each triple of indices out of 4 yields two generators.
We list the representatives on the left and the homology classes on the right.
\begin{eqnarray*} 
T & (A_{12}A_{23})^* \\
(12)T & (A_{12}A_{13})^* \\
(1234)T & (A_{23}A_{34})^* \\
(134)T &  (A_{23}A_{24})^* \\ 
(23)T & (A_{13}A_{34})^* \\  
(13)(24)T & (A_{13}A_{14})^* \\ 
(34)T & (A_{12}A_{24})^* \\
(34)(12)T  & (A_{12}A_{14})^*
\end{eqnarray*}

The three remaining classes are each represented by two independent ``planetary systems'':  the class $(A_{12}A_{34})^* = a \cdot a$
 is represented by a map from the torus describing a configuration space where the star 1 has a planet 2, the star 3 has a planet 4, and the two systems are far apart. 
 
 \
 
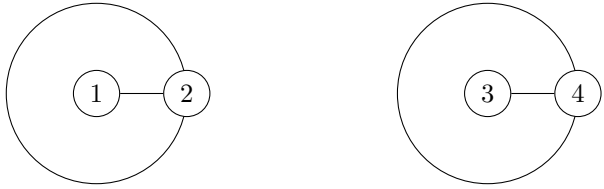
 \begin{tikzpicture} [scale=.4]
  \tikzstyle{star}=[draw,shape=circle,fill=white];
 \draw (0,0)circle (3); 
 \draw (13,0) circle (3);
 \node[star] (1) at (0,0) {1};
  \node[star] (2) at (3,0) {2};
  \node[star] (3) at (13,0) {3};
   \node[star](4) at (16,0) {4};

  \draw (1)--(2)
            (3)--(4);
  \end{tikzpicture}           

\
 
 On the chain level it is represented by $\gamma \cdot \gamma$.
The multiplication in the Barratt Eccles operad 
follows from the operadic structure, as $x \cdot y = (m \circ_2 y) \circ_1 x $, where $m=(1 2)$.   
For example  $$(1 2) \cdot (2 1)= (1 2 3 4 | 2 1 3 4 | 2 1 4 3)+
(1 2 3 4 | 1 2 4 3 | 2 1 4 3),$$ and so in total $\gamma \cdot \gamma$ has 8 summands, namely

\begin{align*} \label{2block}
\gamma \cdot \gamma = (1234|2134|2143)+(1234|1243|2143)+(2134|1234|1243)+(2134|2143|1243)+\\
(1243|2143|2134)+(1243|1234|2134)+(2143|1243|1234)+(2143|2134|1234). 
\end{align*}

This and the two remaining representatives (on the left) of the homology basis generators (on the right) are 
\begin{eqnarray*} 
\gamma \cdot \gamma & (A_{12}A_{34})^* \\
(23) (\gamma \cdot \gamma) & (A_{13}A_{24})^* \\
(24) (\gamma \cdot \gamma) &  (A_{14}A_{23})^*
\end{eqnarray*}

 \begin{Rem} \label{block}
 The representing cycles have a very special form. Namely they involve 2-simplexes where either 
 \begin{itemize}
 \item  two blocks of two labels move alternatively  (as for $\gamma \cdot \gamma$), or  
 \item  in a block of three labels, alternatively one jumps over the other two, or the other two swap (as for $T$). The fourth label is fixed on the right.

 \end{itemize}
 
 Only some labels are allowed. For example the subsequence $(3124|1234)$ does not appear because $3$ is never the star of a planetary system.
This will facilitate evaluation of cocycles on these cycles, as most cochain generators will pair trivially with them.  
  \end{Rem}

We shall compute the value of $\alpha:W_2 \to H^2(F_4)$ on 6 generators of $W_2$ out of 90, namely 
$$(B_{12}B_{23}B_{13})^* , (B_{12}B_{24}B_{14})^* , (B_{12}B_{34}B_{24})^* , (B_{23}B_{13}B_{24})^*, (B_{23}B_{24}B_{14})^* , (B_{23}B_{34}B_{24})^*.$$
We need first to compute the coproduct on these generators, by dualizing the multiplication in the standard basis.

\begin{Lem} \label{coproduct}
\begin{align*}
\mu^*(B_{12}B_{23}B_{13}\,^*)&=(B_{12}B_{13}^* + B_{12}B_{23}^* +B_{23}B_{13}^*) \ot B_{12}^* +   \\
        &+ B_{12}B_{23}^* \ot B_{13}^* +  B_{23}^* \ot B_{12} B_{13}^* + B_{12}^* \ot B_{23}B_{13}^* \\
\mu^*(B_{12}B_{24}B_{14}\,^* )&=( B_{12}B_{14}^*+ B_{12}B_{24}^*+B_{24}B_{14}^*) \ot B_{12}^* + \\
&+ B_{12}B_{24}^* \ot B_{14}^* + B_{24}^* \ot B_{12} B_{14}^* + B_{12}^* \ot B_{24}B_{14}^*  \\
\mu^*(B_{12}B_{34}B_{24}\,^*)&= B_{34}B_{24}^* \ot B_{12}^* + (B_{12}B_{24}^* + B_{12}B_{34}^*)\ot B_{23}^* + B_{12}B_{34}^* \ot B_{24}^* +  \\
&+B_{24}^* \ot B_{12}B_{23}^* + B_{34}^* \ot (B_{12}B_{23}^* + B_{12}B_{24}^*) + B_{12}^* \ot B_{34}B_{24}^* \\
\mu^*(B_{23}B_{13}B_{24}\,^*)&= (B_{13}B_{24}^*+ B_{23}B_{24}^*)  \ot B_{12}^* + B_{23}B_{24}^* \ot B_{13}^* + B_{23}B_{13}^* \ot B_{24}^* + \\
&+B_{13}^* \ot B_{12}B_{24}^* + B_{23}^* \ot (B_{12}B_{24}^* + B_{13}B_{24}^*) + B_{24}^* \ot B_{23}B_{13}^*  \\
\mu^*(B_{23}B_{24}B_{14}\,^*)&= (B_{23}B_{14}^*+B_{23}B_{24}^*)\ot B_{12}^* + B_{24}B_{14}^* \ot B_{23}^* + B_{23}B_{24}^* \ot B_{14}^* +  \\
&+(B_{14}^*+ B_{24}^*) \ot B_{12}B_{23}^* + B_{24}^* \ot B_{23}B_{14}^* + B_{23}^* \ot B_{24}B_{14}^*  \\
 \mu^*(B_{23}B_{34}B_{24}\,^*) &= (B_{23}B_{24}^* + B_{23}B_{34}^*+ B_{34}B_{24}^*) \ot B_{23}^* + B_{23}B_{34}^* \ot B_{24}^* + \\
&+ B_{34}^* \ot B_{23}B_{24}^* + B_{23}^* \ot B_{34}B_{24}^* 
\end{align*}
\end{Lem}

We can then state the computation.

\begin{Prop}
\begin{eqnarray}
\alpha(B_{12}B_{23}B_{13} \,^*) &=& A_{12}A_{13}+A_{12}A_{23}  \label{8} \\  
\alpha(B_{12}B_{24}B_{14} \,^*) &=& A_{14}A_{12} + A_{12}A_{24} \label{18} \\
\alpha(B_{12}B_{34}B_{24}\,^*) &=& 0  \label{22} \\
\alpha(B_{23}B_{13}B_{24}\,^*) &=& 0 \label{39} \\
\alpha(B_{23}B_{24}B_{14}\,^*) &=& 0 \label{49} \\
\alpha(B_{23}B_{34}B_{24}\,^*) &=& A_{23}A_{24} + A_{23}A_{34} \label{57} 
\end{eqnarray}
\end{Prop}

\begin{proof}
We need to compute the map in formula \ref{formulona} on our generators.
The coproduct in Lemma \ref{coproduct}
splits as $\mu^*=\mu^*_{1,0} + \mu^*_{0,1}$ with $\mu^*_{1,0}: W_2 \to W_1 \ot W_0$ and $\mu^*_{0,1}:W_2 \to W_0 \ot W_1$.
We prove equation $\ref{8}$ first.
The first part of the coproduct is
$$\mu^*_{1,0}((B_{12}B_{23}B_{13})^*)=  (B_{12}B_{23})^* \ot B_{13}^* + (B_{23} B_{13}^*+B_{12}B_{13}^*+B_{12}B_{23}^*) \ot B_{12}^*\, .$$
Notice that 
$$\phi_1((B_{12}B_{23})^*)= \omega_{12}\cup_1 \omega_{23} =\pi^*_{123}(1 2 3 | 3 2 1)^*$$
has trivial cup product with $\omega_{12}$ and $\omega_{13}$  by the complexity 2 hypothesis (two indices cannot swap more than once in the sequence of permutations defining a simplex).
Similarly $$\phi_1((B_{12}B_{13})^*) \cup \phi_0(B_{12}^*)=    (\omega_{12} \cup_1 \omega_{13})\cup \omega_{12}=0.$$
 The only class contributing is then $(B_{23}B_{13})^* \ot B_{12}^* \,$, mapping via $\phi_1 \otimes \phi_0$ to $ \pi^*_{123}(Ar) \ot \omega_{12} $.
 The cup product $\pi^*_{123}(Ar) \cup \omega_{12}$  
 pairs non-trivially with 2-simplices $(a|b|c)$ such that the relative position of 1 and 2 is $(12 | 12 | 21)$ by the complexity 2 requirement.
The positions of 1,2,3 in the first two permutations $(a|b)$ must match those of $Ar=(132|312)^*$.
By the restrictions in Remark \ref{block}
the only simplex pairing non trivially is $(a|b|c)=(1324|3124|2314)$ that appears in the expansion of the cycle representative of $(A_{12}A_{13})^* $.

The second part of the coproduct  gives 
$$\mu^*_{0,1}((B_{12}B_{23}B_{13})^*)=  B_{12}^* \ot (B_{23}B_{13})^* + B_{23}^* \ot (B_{12}B_{13})^*.$$ 
 The first summand goes to $\omega_{12} \cup \pi^*_{123}(Ar)$, that is trivial since the positions of $1,2$ in the factors are not compatible.
The second summand goes to $\omega_{23} \cup (\omega_{12} \cup_1 \omega_{13})$. 
The only simplex pairing non trivially with a cycle representative is $(1 2 3 4| 1 3 2 4| 3 2 1 4)$. In fact 2 and 3 must swap first, and then 1 must jump both 2 and 3 by the $\cup_1$ product. This simplex appears in the representative of $(A_{12}A_{23})^*$.

We prove next equation \ref{18}. From now on we suppress the cup product from the notation.

By lemma \ref{coproduct} and formula \ref{formulona}

\begin{align*}
\phi(d(B_{12}B_{24}B_{14})^*)=(\om_{12} \cup_1 \om_{14}) \om_{12}+ (\om_{12}\cup_1 \om_{24}) \om_{12}  
+ \pi_{124}^*(Ar)\om_{12}+ \\
+(\om_{12}\cup_1 \om_{24})\om_{14}+
\om_{24}(\om_{12}\cup_1 \om_{14})+ \om_{12} \pi^*_{124}(Ar). 
\end{align*}

The summands pairing non trivially with cycle simplexes are:

\medskip

 $\pi^*_{124}(Ar)\om_{12}$ pairing with $(1423|4123|2413)$ of $(A_{14}A_{12})^*,$  and
 
\medskip
 
$\om_{24}(\om_{12}\cup_1 \om_{14})$ pairing with $(1243|1423|4213)$ of $(A_{12}A_{24})^*$.

\medskip

We prove equation \ref{22}.
\begin{align*}
\phi(d (B_{12}B_{34}B_{24})^*  )= \pi_{234}^*(Ar)\om_{12}+(\om_{12}\cup_1 \om_{24})\om_{23}+(\om_{12}\cup_1 \om_{34})(\om_{23}+\om_{24})+ \\
+(\om_{24}+\om_{34})(\om_{12} \cup_1 \om_{23})+ \om_{34}(\om_{12}\cup_1 \om_{24})+\om_{12} \pi_{234}^*(Ar)
\end{align*}

and all terms pair trivially with cycle generators.

We prove equation \ref{39}.
\begin{align*}
\phi(d(B_{23}B_{13}B_{24})^*) =  (\om_{13}\cup_1 \om_{24}+\om_{23}\cup_1 \om_{24})\om_{12}+(\om_{23}\cup_1 \om_{24})\om_{13}+ \\
+\pi_{123}^*(Ar)\om_{24}+(\om_{13}+\om_{23})(\om_{12}\cup_1 \om_{24})+\om_{24}\pi^*_{123}(Ar)+\om_{23}(\om_{13}\cup_1 \om_{24}).
\end{align*}

The only terms pairing non trivially are

\medskip

 $\pi_{123}^*(Ar)\om_{24}$ pairing with $(1324|3124|3142)$ of $(A_{13}A_{24})^*$ and
 
 \medskip
  
$\om_{24}\pi^*_{123}(Ar)$ pairing with $(1324|1342|3142)$ also of $(A_{13}A_{24})^*$

\medskip

 so the contributions cancel out.

In order to prove equation \ref{49} we consider the cocycle
\begin{align*}
\phi(d(B_{23}B_{24}B_{14})^*   )= (\om_{23}\cup_1 \om_{14}+\om_{23}\cup_1 \om_{24})\om_{12}+\pi^*_{124}(Ar) \om_{23}+(\om_{23}\cup_1 \om_{24})\om_{14}+ \\
+(\om_{14}+\om_{24})(\om_{12}\cup_1 \om_{23})+\om_{24}(\om_{23}\cup_1 \om_{14})+ \om_{23}\pi^*_{124}(Ar).
\end{align*}

\medskip

Here $\pi^*_{124}(Ar) \om_{23}$ pairs with $(1423|4123|4132)$ of $(A_{14}A_{23})^*$ , and

\medskip
 
$\om_{23}\pi^*_{124}(Ar)$ pairs with $(1423|1432|4132)$ also of $(A_{14}A_{23})^*$.

\medskip

 The contributions cancel out.

We finally prove equation \ref{57} by considering the cocycle.
\begin{align*}
\phi(d(B_{23}B_{34}B_{24})^*)= (\om_{23}\cup_1 \om_{24}+\om_{23}\cup_1 \om_{34}+\pi^*_{234}(Ar))\om_{23}+ \\
+(\om_{23}\cup_1 \om_{34})\om_{24}+\om_{34}(\om_{23}\cup_1 \om_{24})+
\om_{23}\pi^*_{234}(Ar).
\end{align*}

\medskip

Here $\pi^*_{234}(Ar))\om_{23}$ pairs with $(2431|4231|3421)$ of $(A_{23}A_{24})^*$, and

\medskip

$\om_{34}(\om_{23}\cup_1 \om_{24})$ pairs with $(2341|2431|4321)$ of $(A_{23}A_{34})^*$.

\medskip

\end{proof}
 
\section{Non-formality and Hochschild cohomology} \label{sei}

In this section we identify the obstruction class as an element in the Hochschild cohomology of the cohomology algebra of the configuration space.

Let us write $H=H^*(F_k)$ and $W=s(\wt{YB_k})^*$.

The graded vector space $W$ concentrated in degree 1 splits as $W = \oplus_t W_t$, with $W_t$ suspended dual of the vector space generated by words of length $t+1$.

\begin{Def}
The degree 0 homomorphism $\tau:W  \to H$ sending $s(B_{ij})^* \mapsto A_{ij}$ and all other generators to zero is the standard {\em twisting cochain}.
\end{Def}

\begin{Def}
The {\em convolution algebra} structure on  $Hom(W,H)$ uses the coproduct structure on the domain and the product on the range to 
obtain an associative product of degree 1 \cite{BB}. For $f,g:W \to H$, we have the composition
$$f \star g: W \to W \ot W \xrightarrow{f \ot g} H \ot H \to H.$$
\end{Def}

\begin{Lem} \label{startau}
$\tau \star \tau = 0 $. 
\end{Lem}
\begin{proof}
The lemma follows from the coproduct structure described in Lemma \ref{dw2} and the relations holding in the cohomology algebra $H$.
\end{proof}

\begin{Prop} \label{partial}
There is a differential on the convolution algebra 

$$\partial (f)= f \star \tau + \tau \star f$$
\end{Prop}
\begin{proof}
 $\partial$ squares to zero by lemma \ref{startau}.
\end{proof}
The complex $C=Hom(W,H)$ has a bigrading with $C^{p,q}=Hom(W_{p-1},H^{p+q})$, that is inherited by its cohomology  
$H(C^{*,*},\partial).$ 

\begin{Prop} \cite{BB}
There is a bigraded isomorphism  
$$ H(C^{*,*},\partial) \cong HH^{*,*}(H,H)$$
between the cohomology of the convolution algebra and the Hochschild cohomology of the cohomology algebra $H=H^*(F_k)$ with coefficients in itself. 
\end{Prop}
We follow the bigrading convention in \cite{NS} for Hochschild cohomology.
We claim that $\alpha:W_2 \to H^2$ (see Definition \ref{alpha}) represents a non-trivial Hochschild cohomology class of bidegree $(3,-1)$. This will imply that $H$ is not formal. 

\begin{Lem}
$\alpha:W_2 \to H^2$ is a cocycle in the Hochschild complex.
\end{Lem}

\begin{proof}
We must show that $0= \partial(\alpha):W_3 \to H^3(F_4)$.
Given an element $z \in W_3$, consider the differential 
$$dz =\sum_i a_i \cdot b_i + \sum_k c_k \cdot c'_k + \sum_j b'_j \cdot a'_j   \in W_2 \cdot W_0 + W_1 \cdot W_1 + W_0 \cdot W_2 \subset T(W_{\leq 2}).$$
Then $D(dz)-ddz= D(dz)$ is in filtration level 0, namely $D(dz) \in TW_0$, and it has cohomological degree 3.
Since $\phi:(T(W_{\leq 2}),D) \to \mathcal{E}_2^*(4)$ commutes with the differential by Proposition  \ref{prop5.6},
and $\phi^*:W_0 \to H^*(\mathcal{E}_2^*(4))=H^*(F_4)$  coincides with the projection from the bigraded model  by Theorem \ref{222},
it follows that $\phi^*[D(dz)]$ is a trivial cohomology class in $H^3(F_4)$. 
Observe that $D-d$ is trivial on the summand $W_1 \cdot W_1$. Then   
$$D(dz)= \sum_i \iota([\phi da_i]) b_i + \sum_j  b'_j \iota([\phi da'_j]) $$ and 
$\phi(D(dz))$
represents the (trivial) cohomology class  
$$\sum_i \alpha(a_i)  \rho(b_i) + \sum_j \rho(b'_j) \alpha(a'_i)$$
that by definition of $\partial$ in Proposition \ref{partial} is equal to $\partial(\alpha)(z) \in H^3(F_4)$. This shows that $\partial(\alpha)=0$.
\end{proof}

\begin{Lem} \label{impos}
There is an isomorphism $\psi:(TW_{\leq 2},d) \to (TW_{\leq 2},D)$ such that  $\psi-id$ lowers the filtration level
if and only if  $[\alpha] \in HH^2(H^*F_4,H^*F_4)$  is trivial.
\end{Lem}

\begin{proof}
By assumption $\psi=id$ on $W_0$,
 $\psi(W_1) \subseteq W_1 \oplus W_0$ and $\psi(W_2) \subseteq W_2 \oplus W_1 \oplus W_0$, since 
 each $W_i$ has upper degree 1. We write $\psi_i$ for the restriction of $\psi$ to the filtration level $i$.

The value of $\psi$ on $W_1$ is determined by a homomorphism $f:W_1 \to W_0 \cong H^1(F_4)$ such that $\psi |_{W_1}=id+f$. 
Since $d=D$ on $W_1$, and $D(W_0)=d(W_0)=0$, for any such choice $\psi$ obviously commutes with the differentials up to filtration level 1.

Now to make $\psi$ commute with the differential in filtration level 2 we need to construct
$\psi_2(v)=v+w $ , with $w \in W_1 \oplus W_0$.
Then $D\psi_2(v)=Dv+dw$ must be equal to $\psi_1(dv)$.
This is possible if and only if 
the cocycle $Dv-\psi_1(dv)$, for any $v \in W_2$, is a $d$-coboundary.  We have from Definition \ref{iota} 
 that $Dv-\psi_1(dv)= \iota(\alpha(v))+\bar{\partial}(f)(v)$ where 
$\bar{\partial}(f)$ is the composition
$$W_2 \lrh  (W_1 \ot W_0) \oplus (W_0 \ot W_1)  \xrightarrow{f \ot 1 + 1 \ot f} W_0 \ot W_0 \stackrel{\mu}{\lrh} T_2(W_0) .$$
The cohomology class of $[\bar{\partial}(f)(v)]$ is exactly $\partial(f)(v) \in H^2$, with 
$$\partial: Hom^0(W_1,H^1) \rightarrow Hom^1(W_2,H^2)$$ 
as in Proposition \ref{partial}. 
In cohomology $[Dv-\psi_1(dv)]=\alpha(v)+\partial(f)(v)$ vanishes for any $v \in W_2$ if and only if $\psi$ exists.  
So $\psi$ exists if and only if there is some $f$ such that $\partial(f)=\alpha$.

\end{proof}
\
Notice that $dim(W_1)=25, \; dim(H^1)=6, \; dim(W_2)=90, \; 
dim(H^2)=11.$

\medskip

We prove next that $[\alpha]$ is not trivial, by exhibiting a cycle $\beta$ in the dual complex such that $[\beta]$ pairs non-trivially with it. 

\medskip

Let us consider the dual complex $C^*= W \ot H_*$, the tensor product of the dual Yang-Baxter algebra with the homology of the configuration space.
Given the coproduct $\mu^*: W_p \to (W_{p-1} \ot W_0) \oplus ( W_0 \ot W_{p-1})$ we have the composite differential 
\begin{equation} \label{deltabar}
\partial^*: W_p \ot H_q \xrightarrow{ (id \ot \tau \oplus \tau \ot id)  \circ \mu^* } (W_{p-1} \ot H^1 ) \oplus (H^1 \ot W_{p-1})  \ot H_q 
\stackrel{\cap}{\lrh}  W_{p-1} \ot H_{q-1}
\end{equation}

Let us consider the element 

\begin{align*}
\beta &=  B_{12} B_{23} B_{13}\,^* \ot (A_{13} A_{14}\,^* + A_{13} A_{24}\, ^*) +  \\
&+ B_{12}B_{24}B_{14}\,^* \ot  A_{12} A_{14}\,^* + \\
&+ B_{12}B_{34} B_{24}\,^* \ot  (A_{12}A_{13}\,^* + A_{12}A_{23}\,^* + A_{12}A_{14}\,^* + A_{12}A_{24}\,^* ) + \\
&+ B_{23}B_{13}B_{24}\,^* \ot (A_{13}A_{14}\,^* + A_{13}A_{24}\,^*) + \\
&+ B_{23}B_{24}B_{14}\,^* \ot A_{12}A_{14}\,^*  + \\
&+ B_{23}B_{34}B_{24}\,^* \ot A_{12} A_{34}\,^* \\
\end{align*}

\begin{Lem} $\beta \in W_2 \ot H_2(F_4)$ is a cycle with respect to the differential $\partial^*$ defined in equation \ref{deltabar}. 
\end{Lem}
\begin{proof}
We computed the coproduct on the generators of $W_2$ appearing in $\beta$ in Lemma \ref{coproduct}. 
Let us decompose $\beta=\beta_1+\beta_2+\beta_3+\beta_4+\beta_5+\beta_6$ into the six summands appearing in its definition.
The composition of the coproduct with the twisting cochain and the cap product  gives the six summands

\begin{align*}
\partial^*(\beta_1)= B_{12}B_{23}^* &\ot (A_{14}^* + A_{24}^* ) \\
\partial^*(\beta_2)= (B_{12}B_{14}^* &+B_{12}B_{24}^*+ B_{24}B_{14}^*) \ot (A_{14}^*+A_{24}^*) +B_{12}B_{24}^* \ot (A_{12}^*+ A_{24}^*) + \\
+&I(A_{14}^* \ot B_{12}B_{14}^* + (A_{14}^*+ A_{24}^*) \ot B_{24}B_{14}^*)  \\
\partial^*(\beta_3)= B_{34}B_{24}^* &\ot (A_{13}^*+A_{23}^*+A_{14}^*+A_{24}^*) + B_{12}B_{24}^* \ot A_{12}^* + \\
&+B_{12}B_{23}^* \ot A_{12}^* + B_{34}B_{24}^* \ot (A_{13}^*+A_{23}^*+A_{14}^*+ A_{24}^*) \\
\partial^*(\beta_4) = B_{23}B_{24}^* &\ot (A_{14}^*+ A_{24}^*) + B_{23}B_{13}^* \ot A_{13}^* +  \\
&+I((A_{24}^*+A_{14}^*) \ot B_{12}B_{24}^* + A_{13}^* \ot B_{23}B_{13}^*) \\
\partial^*(\beta_5) = B_{23}B_{14}^*& \ot A_{14}^*+B_{23}B_{24}^* \ot A_{14}^*+ B_{23}B_{24}^*\ot (A_{12}^*+A_{24}^*)+ \\
&+I((A_{12}^*+A_{24}^*) \ot B_{12}B_{23}^*+A_{14}^* \ot (B_{12}B_{23}^*+B_{23}B_{14}^*) ) \\
\partial^*(\beta_6)= I(A_{12}^* &\ot B_{23}B_{24}^*)
\end{align*}
where  $I$ indicates interchange of the tensor factors, applied to the summands coming from the second piece of the coproduct. 
The right hand sides add up to zero and so $\partial^*(\beta)$ is a cycle. 
\end{proof}

\begin{Prop} \label{nontrivial}
The class $[\alpha] \in HH^2(H^*(F_4),H^*(F_4))$ is not trivial.
\end{Prop}

\begin{proof}  The cycle $\beta$ pairs non-trivially with the cocycle $\alpha$.
The only non-trivial pairing occurs between the summand  
$$B_{12} B_{24}B_{14}^* \ot A_{12}A_{14}^* $$ of $\beta$
and the summand $$B_{12} B_{24}B_{14} \ot A_{12}A_{14}$$ of $\alpha$. Hence $<\alpha,\beta>=<[\alpha],[\beta]>=1$.
\end{proof}

Now we are ready to prove the main theorem. 

\begin{Prop} \label{4}
$F_4(\R^2)$ is not formal over $\Z_2$.
\end{Prop}

\begin{proof}
If $F_4$ was formal, by Proposition \ref{341} there would be an isomorphism 
$$\psi: (T(W),d) \cong (T(W),D),$$ with $\Psi-id$ lowering filtration level, and then $[\alpha]$ would be trivial by  Lemma \ref{impos}.
But by Proposition \ref{nontrivial} $[\alpha]$ is not trivial, and so $F_4$ is not formal. 
\end{proof}

\begin{Prop} \label{>4}
$F_k(\R^2)$ is not formal over $\Z_2$ for $k>4$.
\end{Prop}
\begin{proof}
Suppose by contradiction that for $k>4$ the configuration space
$F_k$ is formal over $\Z_2$. 
Recall that $H^*(F_k)$ has as Koszul dual the algebra $YB_k$.

The bigraded dual of the bar construction $B (YB_k)^*$ as seen earlier is the bigraded model of  $H^*(F_k)$, and its homology is identified 
in a standard way with it.
If $F_k$ is $\Z_2$-formal, then there is a quasi-isomorphism 
$$q:B(YB_k)^* \stackrel{\simeq}{\longrightarrow} C^*(F_k)$$
 inducing the identity in cohomology
\cite{EH} , since the first algebra is cofibrant in the model structure \cite{Jardine}.

Consider the projection $p: F_k \to F_4$ forgetting all points with labels larger than 4.
It is easy to see that this map admits a section $s:F_4 \to F_k$.

Consider the DGA map $\bar{p}:YB_k \to YB_4$, defined purely algebraically, that sends all generators with indexes larger than 4 to zero.

The homomorphism  $p_\#:P\beta_k \to P\beta_4$ of fundamental groups induces via the descending central series a Lie algebra homomorphism
$p_\flat : \mathcal{L}_k \to \mathcal{L}_4$, that after applying the functor $U$ gives back $\bar{p}$, 
modulo the identification $YB_k \cong  U \mathcal{L}_k$ of  Theorem \ref{descend}.

Consider now the composite homomorphism of DGA's 
$$B(YB_4)^* \xrightarrow{B(\bar{p})^*} B(YB_k)^* \xrightarrow{q}
C^*(F_k) \xrightarrow{s^*} C^*(F_4).$$ 
It induces the identity on cohomology in degree $1$, and henceforth in all degrees. This is a contradiction because $F_4$ is not formal over $\Z_2$.
\end{proof}

Propositions \ref{4} and \ref{>4} together prove Theorem \ref{main} and then Corollary \ref{overz}.

\

Open questions: is  $F_4(\R^2)$ non-formal over $\Z_p$ for some odd prime $p$ ? Does this happen for infinitely many primes?
 How far do we have to go in the filtered model  to find an obstruction? 
The same questions hold for $F_k(\R^2)$ and $k>4$.

Notice that the rational formality of a {\em simply connected}  space implies its $\Z_p$-formality for all primes $p$ but a finite number
(Theorem 3.1 in \cite{EH} ). However this does not apply in our case, since $F_k(\R^2)$ is a (non-nilpotent)  $K(\pi,1)$.

\end{document}